\documentclass{article}
 \usepackage{float}
 \usepackage[T1]{fontenc}
 \usepackage[utf8]{inputenc}
 \usepackage{amsmath}
 \usepackage{amsfonts}
 \usepackage{amssymb}
 \usepackage{graphicx}
\usepackage{verbatim}
\usepackage{ragged2e}
\usepackage{amsmath}
\usepackage{float}
\usepackage{caption}
\usepackage{blindtext}
\usepackage{multicol}
\usepackage{subcaption}
\usepackage{amssymb}
\usepackage{mathrsfs}
\usepackage{euscript}
\usepackage{amsthm}
\usepackage{mathtools}
\numberwithin{equation}{section}
\usepackage{bigints}

\usepackage{scalerel,stackengine}
\usepackage{upref,amsthm,amsxtra,exscale}
\usepackage{nameref}
\usepackage{varioref}
\usepackage[colorlinks=true,urlcolor=blue,
citecolor=red,linkcolor=blue,linktocpage,pdfpagelabels,
bookmarksnumbered,bookmarksopen]{hyperref}
\usepackage{cleveref}
\usepackage{cite}
\usepackage{fullpage}
\usepackage[dvipsnames]{xcolor}
\usepackage{lipsum}
\providecommand\abstractname{Abstract}

\def\cG{\mathcal{G}}
\def\k{\kappa}

\def\n{\mathbb{N}}
\def\eps{\varepsilon}

\def\bar{\overline}

\newcommand{\dsup}{\displaystyle\sup}
\newcommand{\dinf}{\displaystyle\inf}

\renewenvironment{abstract}{%
  \centering\small
  \textbf\abstractname
  \list{}{\leftmargin0.2cm \rightmargin\leftmargin}
  \item\relax
}{%
  \endlist \par\bigskip
}
\advance\oddsidemargin-\textwidth
\evensidemargin=\oddsidemargin

\theoremstyle{plain}
\newtheorem{theorem}{Theorem}[section]
\newtheorem{lemma}[theorem]{Lemma}

\newtheorem{coro}[theorem]{Corollary}

\theoremstyle{definition}
\newtheorem{definition}{Definition}[section]

\newtheorem{remark}{Remark}

\allowdisplaybreaks

\begin{document}

\title{	\vspace{-2cm} \textbf{Higher regularity estimates for solutions to $\infty$-Laplacian-type models} }

\markboth{Higher Estimates of Solutions to an Infinity-Laplacian Elliptic Equation
}
         {Higher estimates of Solutions to an Infinity-Laplacian Elliptic Equation
}

\def\authorOne{\authorfont{Jo\~{a}o Vitor da Silva}}
\def\authorTwo{\authorfont{Makson S. Santos}}
\def\authorThree{\authorfont{Mayra Soares}}
\def\institutionOne{\subauthorfont{Departamento de Matem\'{a}tica - IMECC, Universidade Estadual de Campinas, Campinas, Brazil}}
\def\institutionTwo{\subauthorfont{Departamento de Matem\'{a}tica do Instituto Superior T\'{e}cnico,	Universidade de Lisboa, Lisboa, Portugal
}}
\def\institutionThree{\subauthorfont{Departamento de Matem\'{a}tica - IE/MAT, Universidade de Bras\'{i}lia, Bras\'{i}lia, Brazil}}

\author{Jo\~{a}o Vitor da Silva,\footnote{{Departamento de Matem\'{a}tica - IMECC, Universidade Estadual de Campinas, Campinas, Brazil}} \ \  Makson S. Santos\footnote{Centro de estudos matem\'{a}ticos (CEMS.UL),	Universidade de Lisboa, Lisboa, Portugal
	}\ \  and Mayra Soares\footnote{Departamento de Matem\'{a}tica - IE/MAT, Universidade de Bras\'{i}lia, Bras\'{i}lia, Brazil}}

\maketitle

\setcounter{tocdepth}{3}

\maketitle

\bgroup

\vspace{-0.5cm}
\hrule
\egroup

%
%%  DO NOT EDIT  LINES ABOVE                     %
\bigskip
\begin{abstract}
	\textit{In this work we tackle the higher regularity estimates of solutions to
 inhomogeneous $\infty-$Laplacian equations at interior critical points. Our estimates provide smoothness properties better than the corresponding available regularity for the model with bounded forcing terms. We explore several scenarios, thereby obtaining improved regularity estimates, which depend on the
 universal parameters of the model. Our findings connect with nowadays well-known estimates developed for obstacle and dead-core type problems.}
	\bigskip

	
	
\end{abstract}

\bgroup
%\color{abs}
\hrule
\egroup
\section{Introduction}

\quad  We investigate higher regularity properties at critical points of existing viscosity
solutions to elliptic partial differential equations driven by $\infty-$Laplacian operator in the form
\begin{align}\label{P}
	\left\{
	\begin{array}{lll}
	&	\Delta_\infty u= \mathcal{G}(x, u, Du) \quad \mbox{in} \quad  B_1,  \\  \\ &u \in C^0({B_1}) \quad \mbox{and} \quad \|u\|_{L^{\infty}({B_1})}\leq  1 \,\,\,(\text{normalized solution})
	\end{array}
	\right.
\end{align}
where the source term behaves as 
$$
\mathcal{G}(x, rt, s\xi)\lesssim r^ms^{\kappa}|f(x)||t|^m \min\{1, |\xi|^\kappa\} \quad \text{for every} \quad r, s \in (0, 1],
$$
with the weight function $f \in L^{\infty}(B_1)\cap C^0(B_1)$, and the exponents $m>0$ and $\kappa>0$ satisfying a sort of compatibility conditions given by the set ${\mathcal S}_{m,\kappa}$ below.

Additionally, we define the $1-$nullity set (or critical set of vanishing order $1$)  as follows 
$$
\mathrm{C}_{\mathcal{Z}}(u):=\{x\in B_1 : u(x) = |Du(x)| = 0\}.
$$

In such a context, we are going to prove that, there exists $r \in (0,\frac{1}{2})$, and a universal constant $\mathrm{C}>0$  such that
\begin{align*}
|u(x)|\leq \mathrm{C}\cdot |x-x_0|^{\alpha}\quad \mbox{in} \quad B_r(x_0),
\end{align*}
where
\begin{equation}\label{eq_ouralpha}
\alpha := \dfrac{4-\k}{3-(m+\k)}>1,    
\end{equation}
for every $x_0 \in \mathrm{C}_{\mathcal{Z}}(u).$ Thus, employing this, we can conclude that $u$ is of class $ C^{[\alpha], \alpha - [\alpha]}$ at each point $x_0 \in \mathrm{C}_{\mathcal{Z}}(u)$ under suitable tangential assumptions to be clarified soon.
Note that, for this purpose, the pair $(m,\k)$ must belongs to the following set
%\[\mathcal{S}_{m,\k} := \{ (m,\k) \in (-1,+\infty) \times (-\infty, 4) :
%m<3-\k\} \cup \{ (m,\k) \in (-1,+\infty) \times (4, +\infty) :
%m>3-\k\}  .\]

\[\mathcal{S}_{m,\k} := \{ (m,\k) \in [0,3) \times [0, 4) :
m<3-\k\} 
\]

In this scenario, under a sort of smallness
regimes (of density or magnitude of negativity part of solutions), solutions will enjoy an improved decay along free boundary points.  

\begin{theorem}\label{maintheo_1}
Let $u \in C^0(B_1)$ be a viscosity solution to \eqref{P} with $x_0 \in\mathrm{C}_{\mathcal{Z}}(u)$. There exist $\eps > 0$ such that if one of the following conditions
\begin{itemize}
\begin{multicols}{2}
\item[(i)] $\displaystyle \inf_{B_1} u \geq -\eps,$
\item[(ii)] $\dfrac{|\{u<0\}\cap B_1|}{|B_1|} \leq \eps$,
\end{multicols}
\end{itemize}
holds true, then given $r \in \left(0,\frac{1}{2}\right)$, we can find a universal constant $\mathrm{C}>0$  such that
\begin{align}\label{eq:estimate}
\sup_{B_r(x_0)}|u(x)| \leq \mathrm{C} r^{\alpha},
\end{align}
for $\alpha>1$ as in \eqref{eq_ouralpha}.
\end{theorem}

As an immediate corollary, we obtain a finer growth decay along free boundary points belonging to the weight function's vanishing set $f$.  

\begin{coro} Suppose the assumptions of Theorem \ref{maintheo_1} are in force. Suppose further, $|f(x)|\lesssim \mathrm{c}_0 \mathrm{dist}^{\beta}(x, \mathrm{F})$, for some closed set $\mathrm{F} \subset B_1$, some $\beta>0$, and a universal constant $\mathrm{c}_0>0$. If $x_0 \in \mathrm{C}_{\mathcal{Z}}(u) \cap \mathrm{F}$, then
$$
\sup_{B_r(x_0)}|u(x)| \leq \mathrm{c}_0 r^{\frac{4-\kappa+\beta}{3-(m+\kappa)}}.
$$
\end{coro}

\begin{remark}
Another relevant model for \eqref{P} is the H\'{e}non-type problem with general weights, absorption, and noise terms as follows:
$$
\displaystyle \Delta_{\infty} u (x) =  \sum_{i=1}^{k_0} \mathcal{G}_i(x, u, Du) \lesssim \mathrm{c}_n\sum_{i=1}^{k_0}  \mathrm{dist}^{\beta_i}(x, \mathrm{F}_i) |u|^{m_i}\min\{1, |Du|^{\kappa_i}\} + g_i(|x|) \quad \text{in} \quad B_1,
$$
where $\mathrm{F}_i \subset B_1$ are closed sets, $0\leq m_i+\kappa_i <3$, $\beta_i > 0$, $\mathrm{c}_n>0$ is a universal constant, and
$$
\displaystyle \limsup_{|x| \to x_0 \atop{x_0 \in \mathrm{F}_i} } \frac{g_i(|x|)}{\mathrm{dist}(x, \mathrm{F}_i)^{\sigma_i}} = \mathfrak{L}_i \in [0, \infty) \quad \text{for some} \quad \sigma_i \ge 0.
$$
Therefore, \ along the set \ $\displaystyle \left(\bigcap_{i=1}^{k_0} \mathrm{F}_i\right) \cap \mathrm{C}_{\mathcal{Z}}(u)$, \ the \ viscosity \ solutions belong to $ C_{\text{loc}}^{\alpha_\beta}$, \ where \ 
${\alpha_\beta : ={\displaystyle \min_{1 \le i \le k_0} \left\{\frac{4 + \beta_i-\kappa_i}{3 - (m_i+\kappa_i)}, \frac{4 + \sigma_i}{3}\right\}}}$.
\end{remark}

As a matter of motivation, we may cite recent advancements in \cite{Tei22} and \cite{NSST23} (see, for instance, the literature review in Section \ref{S2} for further details), where the authors have introduced a remarkably effective "iteration of regularity" technique, which has been employed to establish sharp estimates for solutions of \(\mathfrak{F}(x, D^2u) = f(x, u, Du)\). This methodology capitalizes on the properties of the right-hand side, which, through a nonlinear signature, induces an improved decay of solutions near free boundary points.  However, in the context of models governed by the $\infty$-Laplacian, such as \eqref{P}, the direct application of this technique is hindered by the absence of a well-developed regularity theory for such equations (see Section \ref{S2}). Consequently, we adopt a more direct strategy founded on blow-up analysis, renormalization, and by imposing a sort of one-phase tangential regime, leveraging the fact that solutions remain non-negative under smallness regimes (of density or magnitude if negative part of solutions) or by working in algebraic structures where solutions to homogeneous problems enjoy ``better regularity estimates''.

\medskip

For our second main result, we must introduce the following class of functions
\[
\Xi^{\mathrm{H}}_n(B_1) := \left\{ \mathfrak{h} \in C^0(B_1)\cap L^\infty(B_1) \,;\, \Delta_\infty \mathfrak{h}= 0 \mbox{ in }\; B_1 \subset \mathbb{R}^n  \,\,\, \text{and} \,\,\, \mathfrak{h} \in C_{\text{loc}}^{1, \alpha}(B_1)) \,\,\,(\text{for some}\,\,\alpha\in (0, 1))\right\}
\]
and
\begin{equation}\label{eq_alphan}
\alpha^{\mathrm{H}}_n := \sup\left\{\alpha \in [0,1] \,;\, \exists \; \mathrm{C}_{\alpha, n} > 0 \mbox{ s.t. } \|\mathfrak{h}\|_{C^{1,\alpha}(B_1)} \leq \mathrm{C}_{\alpha, n}\|\mathfrak{h}\|_{L^{\infty}(B_1)} \mbox{ for all }\,\, \mathfrak{h} \in \Xi^{\mathrm{H}}_n(B_1)\right\}.    
\end{equation}
For instance, we can infer from Evans-Savin's groundbreaking work \cite{ES08} that $\alpha^{\mathrm{H}}_2 \in (0,1)$ for any $u \in \Xi^{\mathrm{H}}_2(B_1)$ (cf. Savin's manuscript \cite{Savin05} for related results).

\begin{theorem}\label{maintheo_2}
Let $u \in C^0(B_1)$ be a viscosity solution to \eqref{P} with $x_0 \in \mathrm{C}_{\mathcal{Z}}(u)$. Suppose that  $\Xi^{\mathrm{H}}_n \neq \emptyset$. Then, there exists a positive constant $\mathrm{C}_0>0$, such that 
\[
\sup_{B_r(x_0)}|u(x)| \leq \mathrm{C}_0r^{1+\hat{\alpha}},
\]
where    
\begin{equation}\label{eq_alpha2}
\hat{\alpha} \in  \left(0, \alpha^{\mathrm{H}}_n  \right) \cap \left(0, \frac{1+m}{3-(m+\kappa)}\right].
\end{equation}

\end{theorem}

Now, inspired by the ideas in \cite{BF24}, we prove that if the negative part of a solution to \eqref{P} at a point $x_0 \in \Gamma(u):= \{u=0\}\cap \partial\{u>0\}\cap \partial \{u<0\}$ (a branching point) decays as $r^\alpha$ in $B_{r}(x_0)$, for some $\alpha > 0$ as in \eqref{eq_ouralpha}, then the positive part of the solution exhibits the same behavior. More precisely, we aim to prove the following result:

\begin{theorem}\label{ThmPositive-Part}
Let $u \in C^0(B_1)$ be a viscosity solution to \eqref{P}. Consider $x_0 \in \Gamma(u)$ and $r_0 >0$ such that $B_{r_0}(x_0) \subset B_1$. If $\displaystyle \sup_{B_r(x_0)} u^- \leq \mathrm{C}_0r^\alpha$ (resp. $\displaystyle \sup_{B_r(x_0)} u^+ \leq \mathrm{C}_0r^\alpha$) for all $r \in (0,r_0]$ and $\alpha>0$ as in \eqref{eq_ouralpha}, then there exists a positive constant $\mathrm{C}_1$ such that 
\[
\sup_{B_r(x_0)}u^+ \leq \mathrm{C}_1r^\alpha, \,(\mbox{resp. } \sup_{B_r(x_0)} u^- \leq \mathrm{C}_1r^\alpha) \quad \mbox{ for all} \quad r\in (0,r_0].
\]
\end{theorem}

Finally, we will handle another property of the viscosity solutions, namely the non-degeneracy of solutions to \eqref{P}.
In this context, imposing a non-degeneracy assumption on the forcing term is natural. More precisely, we suppose that there are universal constants $\theta>0$, $\sigma>0$ and $0<r_0 \ll 1$ such that
\begin{equation}\label{NDCond}
  \displaystyle \inf_{B_r(x_0)} \frac{\mathcal{G}(x, u, Du)}{r^{\sigma}} > \theta  \quad \text{for all}\quad x_0 \in \mathrm{C}_{\mathcal{Z}}(u) \quad \text{and} \quad \forall \,\,\,r \in (0, r_0).
\end{equation}

In this context, we can prove the following result:

\begin{theorem}[{\bf Non-degeneracy}]\label{ThmNon-Degen}  Let $u \in C^0(\bar{B}_1)$ be a viscosity solution to \eqref{P}. Suppose further that \eqref{NDCond} is in force with $0 < \sigma \leq 3\alpha -4$ (for some $\alpha>\frac{4}{3}$). If $x_0$ is an interior point, then following asymptotic behavior holds
\begin{equation}\label{EqNDeg}
  \displaystyle \limsup_{x \to x_0} \left(\frac{u(x)-u(x_0)}{|x-x_0|^{\alpha}}\right) \ge \sqrt[3]{\frac{\theta}{\alpha^3(\alpha-1)}}.
\end{equation}
In particular, $u \notin C_{\text{loc}}^{\alpha+\epsilon}(B_1)$ for any $\epsilon >0$.
\end{theorem}

\medskip

\begin{remark} We observe that, following a similar strategy, our results can be generalized to the class of problems driven by $(3-\gamma)$-homogeneous operators related to the $\infty$-Laplacian (see \cite{BisVo23} for related topics):  
\begin{align*}
\begin{cases}
	\Delta^\gamma_\infty u = |Du|^{-\gamma} \Delta_{\infty} u =  \mathcal{G}(x, u, Du) \quad \text{in} \quad  B_1,  \\  
	\\ 
	u \in C^0({B_1}) \quad \text{and} \quad \|u\|_{L^{\infty}({B_1})} \leq 1 \quad (\text{normalized solution}),
\end{cases}
\end{align*}
for $\gamma \in [0,2]$. Notice that the operator $\Delta^\gamma_\infty$ corresponds to $\Delta_\infty$ when $\gamma = 0$ (the standard infinity-Laplacian) and to $\Delta^{\mathrm{N}}_\infty$ when $\gamma = 2$ (the normalized infinity-Laplacian). With the appropriate modifications, a result analogous to Theorem \ref{maintheo_1} can be stated as follows:
$$
\sup_{B_r(x_0)} |u(x)| \leq \mathrm{c}_0 \, r^{\frac{4-(\kappa+\gamma)}{3-(m+\kappa+\gamma)}}, 
$$
where $\mathrm{c}_0>0$ is a universal constant.
\end{remark}

The remainder of the manuscript is organized as follows: In Section \ref{S2}, we review the classical literature and recent advances related to our problem. For the sake of completeness, in Section \ref{S3}, we state some well-known auxiliary results, such as the Maximum Principle, $L^\infty$ bounds, and Lipschitz regularity, among others. Theorems \ref{maintheo_1} and \ref{maintheo_2} are proved as a consequence of several auxiliary lemmas presented in Sections \ref{S4} and \ref{S5}, respectively. Section \ref{S6} is dedicated to showing that controlling the growth of the positive (or negative) part of a viscosity solution to Problem \eqref{P} is sufficient to control the entire solution. Finally, in Section \ref{S7}, we establish the proof of the non-degeneracy of solutions to Problem \eqref{P}.

\section{Literature  review}\label{S2}

\subsection{Theory of existence/uniqueness and regularity to infinity-Laplacian}

The mathematical study of problems involving the $\infty$-Laplacian operator has its roots in the seminal works of Aronsson \cite{Aronsson67, Aronsson68}. Aronsson's research primarily addressed the following problem: given a bounded domain $\Omega \subset \mathbb{R}^n$ and a Lipschitz function $g: \partial \Omega \to \mathbb{R}$, determine the optimal Lipschitz extension $\mathcal{G}_b$ of $g$, which satisfies the boundary condition $\mathcal{G}_b = g$ on $\partial \Omega$ and minimizes the Lipschitz constant in every subdomain. More precisely, $\mathcal{G}_b$ must satisfy:
\[
\text{for any } \Omega' \subset \Omega, \text{ if } \mathcal{G}_b = h \text{ on } \partial \Omega', \text{ then } \|\mathcal{G}_b\|_{C^{0, 1}(\Omega')} \leq \|h\|_{C^{0, 1}(\Omega')}.
\]
This function $\mathcal{G}$ is now known as the {\it Absolutely Minimizing Lipschitz Extension} (AMLE) of $g$ in $\Omega$.

Later, Jensen established a key result connecting AMLEs to the $\infty$-Laplacian in the viscosity sense \cite{Jensen93}:
\[
u \text{ is an AMLE} \quad \iff \quad \Delta_{\infty} u(x) = 0 \text{ in } \Omega.
\]
This result revealed that the $\infty$-Laplacian serves as the Euler–Lagrange equation for the $L^\infty$ minimization problem. Jensen also proved the uniqueness of viscosity solutions for the $\infty$-Laplacian with given Dirichlet boundary data (see \cite{ArmSmart10} for an elementary proof of Jensen’s theorem).

Beyond this fundamental work, infinity-harmonic functions and their generalizations arise in various applications across pure and applied mathematics. For instance, the value of certain "Tug-of-War" games corresponds to infinity-harmonic profiles, as explored in \cite{PSSW09} (see also \cite{BEJ08} for an extensive survey).

The regularity theory of $\infty$-Laplacian equations remains an active and challenging area of research. A longstanding conjecture posits that viscosity solutions of
\[
\Delta_{\infty} u(x) = 0 \quad \text{in } \Omega
\]
belong to the H\"{o}lder space $C^{1, \frac{1}{3}}$. However, this conjecture is yet to be proven, despite notable progress in specialized cases, such as infinity-obstacle problems \cite{RTU15} and infinity-dead core problems \cite{daSRosSal19}.

Explicit solutions by Aronsson (see \cite{Aronsson84}) provide insights into the expected regularity. By way of exemplification, the family of solutions
\[
u(x) = \mathfrak{a}_1|x_1|^{\frac{4}{3}} + \cdots + \mathfrak{a}_n|x_n|^{\frac{4}{3}}, \quad \text{where } \sum_{i=1}^{n} \mathfrak{a}_i^3 = 0,
\]
shows that the first derivatives of $u$ are H\"{o}lder continuous with a sharp exponent $\alpha = \frac{1}{3}$, while second derivatives fail to exist along the coordinate axes. In higher dimensions, solutions with discontinuous second derivatives have also been observed (see \cite{Yu06}). Aronsson's pioneering work culminated in the following theorem:

\begin{theorem}[{\bf Aronsson \cite{Aronsson68}}]
Let $u \in C^2(\Omega)$, where $\Omega \subset \mathbb{R}^n$ is a domain. If
\[
\Delta_{\infty} u(x) = 0 \quad \text{in } \Omega,
\]
then $|Du| \neq 0$ in $\Omega$, unless $u$ is constant.
\end{theorem}

Recent advancements by Evans and Savin \cite{ES08} established that infinity-harmonic functions in two dimensions are $C^{1, \alpha_0}$, for some $\alpha_0 \in (0, 1)$, but the optimal value of $\alpha_0$ remains unknown. In higher dimensions, Evans and Smart \cite{ESmart11} demonstrated the universal differentiability of solutions.

For inhomogeneous $\infty$-Laplacian equations of the form
\[
\Delta_{\infty} u(x) = f(x) \quad \text{in } \Omega,
\]
Lu and Wang \cite{LuWang08} proved the existence and uniqueness of continuous viscosity solutions under sign constraints on $f$ (either $\displaystyle \inf_{\Omega} f > 0$ or $\displaystyle \sup_{\Omega} f < 0$). Moreover, uniqueness may fail if this condition is not satisfied (see, e.g., \cite[Appendix A]{LuWang08}). We must quote that Bhattacharya and Mohammed in \cite{BhatMoh11} studied the existence and nonexistence of viscosity solutions to the Dirichlet problem  
$$
\left\{
\begin{array}{rclcl}
\Delta_{\infty} u(x) & = & f(x,u) & \text{in } & \Omega, \\
u(x) & =  & g(x) & \text{on } & \partial\Omega,
\end{array}
\right.
$$
where $f$ satisfies certain sign and monotonicity restrictions. In the sequence, in \cite{BhatMoh12}, Bhattacharya and Mohammed introduced structural conditions on $f$ and established existence results without imposing sign or monotonicity restrictions (See, Section \ref{S3} for more details). 

In conclusion, Lipschitz and differentiability estimates are the only regularity results available for this class of equations, and higher regularity remains an open problem up to this moment (see \cite{Lind14}).

\subsection{Elliptic diffusion models with free boundaries}

In the following, we will discuss the connections and motivations behind our problem \eqref{P} and several key mathematical models that arise in the context of free boundary problems and nonlinear diffusion processes.

A primary example of a free boundary model related to our study is the zero-obstacle problem, which was treated by Rossi \textit{et al.} in \cite{RTU15}. This problem involves finding a function that satisfies the following equation in the viscosity sense:

\[
\Delta_{\infty} u = f(x) \quad \text{in} \quad \{ u > 0 \}, \quad u \geq 0 \quad \text{in} \quad B_1.
\]

Alternatively, the zero-obstacle problem can be formulated as:

\begin{equation}\label{EqIOP}
\min \left\{ \Delta_{\infty} u(x) - f(x), u(x) \right\} = 0 \quad \text{in} \quad B_1,
\end{equation}

which is understood in the viscosity sense. In this framework, the authors provide the following results regarding existence and regularity.

\begin{theorem}[{\bf \cite[Theorem 3.1]{RTU15}}]
Consider a positive function $g \in C(\partial B_1)$, and let $f$ be such that
\[
0 < \mathfrak{m}_0 < f(x) \le \mathrm{M} < \infty.
\]
Then there exists a unique function $u \in C(\overline{B}_1)$ such that
\[
\left\{
\begin{array}{rclcl}
  \min \left\{ \Delta_{\infty} u(x) - f(x), u(x) \right\} & = & 0 & \text{in} & B_1,\\
  u(x) & = & g(x) & \text{on} & \partial B_1
\end{array}
\right.
\]
in the viscosity sense. Moreover, if $f$ is uniformly Lipschitz continuous in $B_1$, then $u$ is locally Lipschitz continuous in $B_1$.
\end{theorem}

The authors also provide the optimal $C^{1, \frac{1}{3}}$ regularity estimate for solutions to the $\infty$-obstacle problem along the free boundary.

\begin{theorem}[{\bf \cite[Theorem 3.3]{RTU15}}]
Let $u \in C(B_1)$ be a solution to \eqref{EqIOP} and let $x_0 \in \partial \{ u > 0 \}$ be a free boundary point. Then,
\[
\sup_{B_r(x_0)} u(x) \leq \mathrm{C} r^{\frac{4}{3}},
\]
for some universal $\mathrm{C} > 0$.
\end{theorem}

Additionally, they show that viscosity solutions leave the free boundary in a $C^{1, \frac{1}{3}}$ manner.

\begin{theorem}[{\bf \cite[Theorem 3.5]{RTU15}}] Let $u$ be a viscosity solution to \eqref{EqIOP} and let $y_0 \in \{ u > 0 \}$ be a point in the closure of the non-coincidence set. Then,
\[
\sup_{B_r(y_0)} u \geq \mathrm{C}(\mathfrak{m}_0) r^{\frac{4}{3}},
\]
for some universal constant $\mathrm{C} > 0$.
\end{theorem}

The second relevant model we highlight is the $\infty$-dead core problem. In \cite{ALT16}, Ara\'{u}jo \textit{et al.} studied reaction-diffusion models governed by the $\infty$-Laplacian operator. For $\lambda > 0$, $0 \leq \gamma < 3$, and $0 < \phi \in C(\partial \Omega)$, let $\Omega \subset \mathbb{R}^n$ be a bounded open domain, and define the operator

\begin{equation}\label{EqDeadCore}
    \mathcal{L}^{\gamma}_{\infty} v := \Delta_{\infty} v - \lambda \cdot (v_+)^{\gamma} = 0 \quad \text{in} \quad \Omega \quad \text{and} \quad v = \phi \quad \text{on} \quad \partial \Omega.
\end{equation}

Here, $\mathcal{L}^{\gamma}_{\infty}$ represents the $\infty$-diffusion operator with $\gamma$-strong absorption, and $\lambda > 0$ is the \textit{Thiele modulus}, which controls the ratio between the reaction and diffusion-convection rates. We refer the reader to \cite{Diaz85, DT20}.

An important feature of equation \eqref{EqDeadCore} is the potential presence of plateaus, subregions where the function becomes identically zero. This phenomenon is explored in \cite{daSRosSal19}, \cite{daSSal18}, and \cite{Diaz85} for similar quasi-linear dead-core problems.

After proving the existence and uniqueness of viscosity solutions using Perron’s method and the Comparison Principle (see \cite[Theorem 3.1]{ALT16}), the main result of \cite{ALT16} ensures that a viscosity solution is locally of class $C^{\frac{4}{3-\gamma}}$ along the free boundary of the non-coincidence set, i.e., $\partial \{ u > 0 \}$ (see \cite[Theorem 4.2]{ALT16}). This implies that solutions grow as $ \mathrm{dist}^{\frac{4}{3-\gamma}} $ from the free boundary. Using barrier functions, the authors show that this estimate is sharp, meaning that $u$ separates from its coincidence region exactly as $ \mathrm{dist}^{\frac{4}{3-\gamma}} $ (see \cite[Theorem 6.1]{ALT16}).

Furthermore, the authors prove some Liouville-type results. Specifically, if $u$ is an entire viscosity solution to
\[
\Delta_{\infty} u(x) = \lambda u^{\gamma}(x) \quad \text{in} \quad \mathbb{R}^n,
\]
with $u(0) = 0$, and $u(x) = \text{o}\left(|x|^{\frac{4}{3-\gamma}}\right)$, then $u \equiv 0$. See \cite[Theorem 4.4]{ALT16}.

We also mention Teixeira's work in \cite{Tei22}, where the author derived regularity estimates for interior stationary points of solutions to $p$-degenerate elliptic equations in an inhomogeneous medium:

\[
\mathrm{div}\,\mathfrak{a}(x, \nabla u) = f(|x|, u) \lesssim \mathrm{c}_0 |x|^{\alpha} |u|^{m},
\]
for some $\mathrm{c}_0 > 0$, $\alpha \geq 0$, and $0 \leq m < p - 1$. In this scenario, the vector field $\mathfrak{a}: B_1 \times \mathbb{R}^n \to \mathbb{R}^n$ is $C^1$-regular in the gradient variable and satisfies the structural conditions:

\begin{equation} \label{condestr}
\left\{
\begin{array}{rclcl}
|\mathfrak{a}(x,\xi)| + |\partial_{\xi}\mathfrak{a}(x,\xi)||\xi| & \leq & \Lambda |\xi|^{p-1}, & & \\
\lambda |\xi|^{p-2}|\eta|^2 & \leq & \langle \partial_{\xi}\mathfrak{a}(x,\xi)\eta, \eta \rangle & & \\
\displaystyle \sup_{x, y \in B_1 \atop{x \ne y, \,\,\,|\xi| \ne 0 }} \frac{|\mathfrak{a}(x,\xi)-\mathfrak{a}(y,\xi)|}{\omega(|x-y|)|\xi|^{p-1}} & \leq & \mathfrak{L}_0 < \infty. & &
\end{array}
\right.
\end{equation}

For this context, Teixeira derived a quantitative non-degeneracy estimate, showing that solutions cannot be smoother than $C^{p'}$ at stationary points (see \cite[Proposition 5]{Tei22}). Additionally, at critical points where the source vanishes, higher-order regularity estimates are given, which are sharp in terms of the vanishing rate of the source term (see \cite[Theorem 3]{Tei22}).

Finally, we relate our study to elliptic equations of the form:
\begin{equation}\label{EqMatukuma}
\mathrm{div}\left(\mathfrak{a}(|x|) |\nabla u|^{p-2} \nabla u\right) = \mathfrak{h}(|x|) f(u) \quad \text{in } \quad \Omega \subset \mathbb{R}^n, \quad  p > 1,
\end{equation}
where $\mathfrak{a}, \mathfrak{h} : \mathbb{R}^+ \to \mathbb{R}^+$ are radial profiles belonging to $C^1(\mathbb{R}^+)$ and $C^0(\mathbb{R}^+)$, respectively. Notably, the celebrated Matukuma equation (or Batt–Faltenbacher–Horst equation, see \cite{BFH86}) serves as a prototype for \eqref{EqMatukuma}.

In this context, the function \(f\) satisfies the following conditions:
\begin{itemize}
    \item[\textbf{(F1)}] $f \in C^0(\mathbb{R})$;
    \item[\textbf{(F2)}] $f$ is non-decreasing on $\mathbb{R}$, and $f(t) > 0$ if and only if $t > 0$.
\end{itemize}

Consider, as a specific example, the model equation:
\begin{equation}\label{ModelEq}
    \mathrm{div} \left( |x|^k | \nabla u|^{p-2} \nabla u \right) = |x|^{\alpha} f(u),
\end{equation}
where $f(u) = u_{+}^{m}$ and the parameters satisfy $m + 1 < k - \alpha < p $.

In this setting, da Silva \textit{et al.} \cite{daSdosPRS} established the following sharp estimate for weak solutions to \eqref{ModelEq}:
\[
\sup_{x \in B_r(x_0)} u(x) \leq \mathrm{C} r^{1 + \frac{1 + \alpha + m - k}{p - 1 - m}},
\]
where $\mathrm{C} > 0$ is a universal constant, $x_0 \in B_1$ is a free boundary point for \( u \), and the nonlinearity satisfies $f(|x|, u) \lesssim |x - x_0|^{\alpha} u_+^m$, with $\alpha + 1 + m > k$.

Recently, Bezerra J\'{u}nior \textit{et al} in  \cite{BeJDaSNS2024} established sharp and improved regularity estimates for non-negative viscosity solutions of H\'{e}non-type elliptic equations, governed by the infinity-Laplacian, under a strong absorption condition:
\begin{equation}\label{pobst}
	\Delta_{\infty}  u(x)  =  f(|x|, u(x)) \quad \text{in} \quad B_1,
\end{equation}
where  for all $(x, t) \in B_1 \times \mathfrak{I}$, $r, s \in (0, 1)$ ($\mathfrak{I} \subset \mathbb{R}$ an interval), the authors assume that there exists a universal constant $\mathrm{c}_n > 0$, and some function $f_0 \in L^{\infty}(B_1)$ such that for every $r, s \in (0, 1)$,
\begin{equation}\label{EqHomog-f}
|f(r|x|, s t)| \leq \mathrm{c}_n r^{\alpha} s^m |f_0(x)| \quad \text{for} \quad 0 \leq m < 3 \quad  \text{and} \quad  \alpha \in \left[{0}, \infty\right).
\end{equation}
A toy model for \eqref{pobst} is the H\'{e}non-type problem, with strong absorption:
$$
\Delta_{\infty} u(x) = \sum_{i=1}^{k_0} c_i |x|^{\alpha_i} u_{+}^{m_i}(x) \quad \text{in} \quad B_1,
$$
where
$$
0 < \alpha_i < \infty, \quad c_i \geq 0, \quad \text{and} \quad 0 \leq m_i < 3 \quad \text{for} \quad 1 \leq i \leq k_0.
$$
In such a context, they addressed the following sharp and improved regularity estimate along free boundary points
\begin{theorem}[{\bf Higher regularity estimates}]\label{Hessian_continuity}
Let $u \in C^0(B_1)$ be a non-negative viscosity solution of problem \eqref{pobst}. Given $x_0 \in B_{1/2}$, assume that $f(|x|, u) \lesssim |x - x_0|^{\alpha} u_+^m(x)$ with \eqref{EqHomog-f} in force. If $x_0 \in B_{1/2} \cap \partial \{ u > 0 \}$, then
\begin{equation}\label{Higher Reg}
u(x) \leq \mathrm{C}\cdot \|u\|_{L^{\infty}(B_1)}|x-x_0|^{\frac{4 + \alpha}{3 - m}}
\end{equation}
for every $x \in \{u>0\} \cap B_{1/2}$, where $\mathrm{C} > 0$ depends on universal parameters.
\end{theorem}
Additionally, they proved the following  non-degeneracy estimate
\begin{theorem}[{\bf Non-degeneracy at critical points}]\label{NãoDeg}
Let $m \in [0, 3)$ and $u \in C^0(B_1)$ be a viscosity solution to problem \eqref{pobst}. Then, there exists $r^{\ast} > 0$ such that for every critical point $x_0 \in B_1$ and for all $r \in (0, r^{\ast})$ such that $B_r(x_0) \subset B_1$, we have
$$
\sup_{\partial B_r(x_0)}  u(x) \geq \left(\frac{(3-m)^4}{(4+\alpha)^3(1+\alpha+m)}\right)^{\frac{1}{3-m}} \cdot r^{\frac{4+\alpha}{3-m}}.
$$
\end{theorem}

We also mention the pioneer manuscript \cite[Theorem 2]{Tei16}, where Teixeira investigated geometric regularity estimates to solve dead-core problems. The model equation under consideration is given by
\begin{equation}
    F(x,D^2u) = (u_+)^{\mu}, \quad 0 \leq \mu < 1, \quad x \in B_1,
\end{equation}
where \( F: B_1 \times \text{Sym}(n) \to \mathbb{R} \) is a fully nonlinear, uniformly elliptic operator with measurable coefficients. That is, there exist constants \( 0 < \lambda \leq \Lambda< \infty \) such that, for all \( \mathrm{X}, \mathrm{Y} \in \text{Sym}(n) \) with \( \mathrm{Y} \geq 0 \), the following inequality holds
\begin{equation}
    \lambda \| \mathrm{Y} \| \leq F(x, \mathrm{X}+\mathrm{Y}) - F(x, \mathrm{X}) \leq \Lambda \| \mathrm{Y} \|, \quad \forall x \in B_1.
\end{equation}
Precisely, for a given non-negative and bounded viscosity solution \( u \), and for any \( x_0 \in \partial \{ u > 0 \} \cap B_{1/2} \), it is established that
\begin{equation}
    u(x) \leq \mathrm{C}_0 \| u \|_{L^{\infty}(B_1)} |x - x_0|^{\frac{2}{1-\mu}}, \quad \forall x \in \{ u > 0 \} \cap B_{1/2},
\end{equation}
where the constant \( \mathrm{C}_0>0 \) depends only on \( n, \lambda, \Lambda \) and \( \mu \).
More precisely Teixeira demonstrated that compactness arguments leading to \( F \)-harmonic functions alone are inadequate to derive any further enhanced regularity estimates. However, when coupled with the non-negativity property, these arguments offer profound insights into the behavior of solutions near free boundary points, leading to an unforeseen improvement in smoothness.

Moreover, our manuscript is motivated by \cite[Theorem 1]{NSST23}, wherein the authors establish superior regularity properties for solutions to fully nonlinear elliptic models of the form
$$
    F(x, D^2u) = f(x, u, Du) \lesssim q(x)|u|^m \min \{1, |Du|^{\gamma}\}, \quad (m, \gamma \geq 0),
$$
at interior critical points, where $q \in L^p(B_1)$ is a non-negative function and $p > n$. The principal innovation of these estimates lies in their ability to confer smoothness properties that exceed the intrinsic regularity constraints imposed by the heterogeneity of the problem.

Below, we summarize the sharp regularity estimates available in the literature for problems closely related to \eqref{P}:

\begin{table}[h]
\centering
\resizebox{\textwidth}{!}{
 \begin{tabular}{c|c|c|c}
{\bf Model PDE} & {\bf Structural assumptions} & {\bf Sharp regularity estimates} & \textbf{References} \\
\hline
$\min \left\{ \Delta_{\infty} u(x) - f(x), u(x) \right\} = 0$ & $f \in L^{\infty}(B_1)$, \,\, $u \geq 0$ & $C_\mathrm{loc}^{1, \frac{1}{3}}$ & \cite{RTU15}  \\\hline
$\Delta_{\infty} u(x) = \lambda u^{\gamma}(x)$ & $\lambda > 0$, \,\, $\gamma \in [0, 3)$ & $C_\mathrm{loc}^{\frac{4}{3-\gamma}}$ & \cite{ALT16}, \cite{DT20} \\\hline
$\mathrm{div} \left( |x|^k | \nabla u|^{p-2} \nabla u \right) = |x|^{\alpha}u_+^m$ & $0 \leq m < p-1$, \,\, $\alpha + 1 + m > k$, \,\, $p > 1$ & $C_\mathrm{loc}^{\frac{p+\alpha-k}{p-1-m}}$ & \cite{daSdosPRS} \\
\hline
$\Delta_p u = f(x, u) \lesssim \mathrm{c}_0|x|^{\alpha}|u|^{m}$ & $\alpha \geq 0$, \,\, $0 \leq m < p-1$ (for $p > 2$) & $C_\mathrm{loc}^{1, \min\left\{\alpha_\mathrm{H}, \frac{p+\alpha+1}{p-1-m}\right\}^{-}}$ & \cite{Tei22} \\
\hline
$\mathcal{G}(x, D^2 u) = (u_+)^{\mu}$ & $\mathcal{G}$ \,\,\text{uniformly elliptic and}\,\,$\mu \in [0, 1)$ & $C_
{\text{loc}}^{\frac{2}{1-\mu}}$ & \cite{Tei16}  \\
\hline
$ F(x, D^2u) = f(x, u, Du) \lesssim q(x)|u|^m \min \{1, |Du|^{\gamma}\}$ & $F$ \text{unif. ellip.}\,\,\,$g \in L^p(B_1)$, \,\,$p>n$\,\,\,\text{and}\,\,\,$m, \gamma\geq 0$ & $C_
{\text{loc}}^{1 + \epsilon_{m, n, p, \gamma}}$ & \cite{NSST23} \\
\hline
$ \Delta_{\infty} u(x) = \mathcal{G}(x, u) \lesssim |x|^{\alpha} u_{+}^m $ & $\alpha, m >0$ & $C_
{\text{loc}}^{\frac{4-\kappa}{3-m}}$ & \cite{BeJDaSNS2024}
\end{tabular}}
\caption{Sharp regularity estimates for elliptic problems related to \eqref{P}.}
\end{table}

Our work aims to enhance and extend some of these results by employing alternative strategies and techniques.

\section{Tools and auxiliary results} \label{S3}

In this section, we gather some fundamental tools from the theory of viscosity solutions to the $\infty$-Laplacian.

\begin{definition}[{\bf Viscosity solution}]
We say that a function $u \in C^0(\overline{\Omega})$ is a viscosity subsolution (resp. viscosity a supersolution) to
\[
\Delta_{\infty} u = \mathcal{G}(x, u, Du) \quad \text{in} \quad \Omega,
\]
if for every test function $\varphi \in C^2(\overline{\Omega})$ and $x_0 \in \Omega$ such that $u - \varphi$ has a local maximum (resp. local minimum) at some $x_0$, the following inequality holds
\[
\Delta_{\infty} \varphi(x_0) \geq \mathcal{G}(x_0, u(x_0), Du(x_0)) \quad \text{(resp. } \,\,\leq \mathcal{G}(x_0, u(x_0), Du(x_0))).
\]
A function $u \in C(\overline{\Omega})$ is a viscosity solution to $\Delta_{\infty} u = \mathcal{G}(x, u, Du)$ if it is both a viscosity subsolution and a viscosity supersolution.
\end{definition}

The following maximum-minimum principle is a fundamental tool in studying viscosity solutions. Alternatively, it can also follow from a Harnack inequality as in \cite{ACJ04}.

\begin{lemma}[{\cite[Lemma 2.1]{BhatMoh11}}]\label{MP}
Let $u \in C^0(\overline{\Omega})$ satisfy $\Delta_{\infty} u \geq 0$
(resp. $\Delta_{\infty} u \leq 0$) in the viscosity sense. Then,
\[
\sup_{\overline{\Omega}} u = \sup_{\partial \Omega} u \quad \left(\text{resp. } \inf_{\overline{\Omega}} u = \inf_{\partial \Omega} u \right).
\]
Moreover, unless $u$ is constant, the supremum (resp. infimum) is attained only on the boundary
$\partial \Omega$.
\end{lemma}

Consider the following Dirichlet problem,
% We now turn our attention to the Dirichlet problem {\color{red} do we really need this part?}
\begin{equation}\label{DirichetProb}
\left\{
\begin{array}{ccc}
    \Delta_{\infty} u(x) & = & f(x, u) \quad \text{in } \Omega, \\
     u(x) & = & g(x) \quad \text{on } \partial \Omega.
\end{array}
\right.
\end{equation}
The next result is a Comparison Principle for solutions to \eqref{DirichetProb}.

\begin{lemma}[{\cite[Lemma 4.1]{BhatMoh11}}]\label{LemmaCP}
Suppose $f_i: \Omega \times \mathbb{R} \to \mathbb{R}$, $i = 1, 2$, are continuous. Let $u, v \in C^0(\overline{\Omega})$ be such that
\[
\Delta_{\infty} u = f_1(x, u) \quad \text{and} \quad \Delta_{\infty} v = f_2(x, v).
\]
Suppose further that either $f_1(x, t)$ or $f_2(x, t)$ is non-decreasing in $t$, and $f_1(x, t) > f_2(x, t)$ for all $(x, t) \in \Omega \times \mathbb{R}$. If $u \leq v$ on $\partial \Omega$, then $u \leq v$ in $\Omega$.
\end{lemma}

\begin{coro}[{\cite[Corollary 4.6]{BhatMoh11}}]
Let $f: \Omega \times \mathbb{R} \to \mathbb{R}$ be continuous and either positive or negative.
If $f$ is also non-decreasing in the second variable, then the Dirichlet problem \eqref{DirichetProb} has at most one solution.
\end{coro}

Next, we present a stability result that ensures convergence under certain conditions.

\begin{lemma}[{\cite[Lemma 5.1]{BhatMoh11}}]\label{stability}
Let $\{f_k\}_{k=1}^{\infty}$ be a sequence of non-negative functions in $C^0(\overline{\Omega})$ such that $f_k \to f$ locally uniformly in $\Omega$, for some $f \in C^0(\overline{\Omega})$. Suppose that for each positive integer $k$, $u_k \in C^0(\overline{\Omega})$ is a solution of the Dirichlet problem:
\[
\left\{
\begin{aligned}
    \Delta_{\infty} u_k & = f_k \quad \text{in } \Omega, \\
    u_k &= g \quad \text{on } \partial \Omega,
\end{aligned}
\right.
\]
and that $u_0 \leq u_k \leq u_1$ in $\Omega$, for some functions $u_0, u_1 \in C^0(\overline{\Omega})$ satisfying $u_0|_{\partial \Omega} = u_1|_{\partial \Omega} = g$. Then, the sequence $\{u_k\}$ has a subsequence that converges locally uniformly in $\Omega$ to a solution $u \in C^0(\overline{\Omega})$ of the Dirichlet problem:
\[
\left\{
\begin{aligned}
    \Delta_{\infty} u & = f \quad \text{in } \Omega, \\
    u & = g \quad \mbox{ on } \partial \Omega.
\end{aligned}
\right.
\]
\end{lemma}

To proceed, we assume that $f: \Omega \times \mathbb{R} \to \mathbb{R}$ satisfies the following boundedness condition: for every compact interval $\mathrm{I} \subset \mathbb{R}$,
\begin{equation}\label{EqCond_f}
\sup_{\Omega \times \mathrm{I}} |f(x, t)| < \infty.
\end{equation}

\begin{theorem}[{\bf A priori $L^{\infty}$ bounds - \cite[Theorem 5.3]{BhatMoh12}}]
Let $\Omega \subset \mathbb{R}^n$ be a bounded domain, $g \in C^0(\partial \Omega)$, and $f \in C^0(\Omega \times \mathbb{R}, \mathbb{R})$ satisfying \eqref{EqCond_f}. Assume further that\\
\begin{equation}\label{Growth_Cond_f}
\left\{
\begin{array}{ll}
    \text{(i)} & \quad \displaystyle \liminf_{t \to \infty} \inf \frac{f(x, t)}{t^3} = \mathfrak{L}_{-}, \\
    \text{(ii)} & \quad \displaystyle \liminf_{t \to -\infty} \sup \frac{f(x, t)}{t^3} = \mathfrak{L}_{+},
\end{array}
\right.
\end{equation}
for some $\mathfrak{L}_{\pm} \in [0, \infty]$. Then, there exists a constant $\mathrm{C}>0$, depending on $f$, $g$, and $\text{diam}(\Omega)$, such that
\[
\| u \|_{L^\infty(\Omega)} \leq \mathrm{C}
\]
for any solution $u \in C^0(\Omega)$ of the Dirichlet problem \eqref{DirichetProb}.
\end{theorem}

We now state an existing result of viscosity solutions for a general class of inhomogeneous problems driven by infinity-Laplacian, which one relates to the model \eqref{P}.

\begin{theorem}[{\bf Existence of solutions - \cite[Theorem 5.5]{BhatMoh12}}] Let $\Omega \subset \mathbb{R}^n$ be a bounded domain, and $g \in C^0(\partial \Omega)$. If $f \in C^0(\Omega \times \mathbb{R}, \mathbb{R})$ satisfies the conditions in \eqref{EqCond_f} and \eqref{Growth_Cond_f}, then the Dirichlet problem \eqref{DirichetProb} admits a solution
$u \in C^0(\Omega)$.

\end{theorem}

The next theorem provides a straightforward version of local Lipschitz regularity for the inhomogeneous problem.

\begin{theorem}[{\bf Local Lipschitz Regularity - \cite[Corollary 2]{Lind14}}]\label{Reg_Lipsc} 
Let \( u \in C^0(B_1) \) be a solution of 
\[
\Delta_{\infty} u = f \in L^{\infty}(B_1).
\]
Then, \( u \) is locally Lipschitz, and in particular,
\[
\| u \|_{C^{0,1}(B_{1/2})} \leq \mathrm{C} \cdot \left( \| u \|_{L^{\infty}(B_1)} + \| f \|_{L^{\infty}(B_1)}^{\frac{1}{3}} \right).
\]
\end{theorem}

\section{Regularity estimates under ``smallness regimes''} \label{S4}

In this section, we improve the regularity of solutions to \eqref{P} under the condition that the solutions are not too negative, or the measure of the set where the solution is negative is sufficiently small. Throughout this section, $\alpha$ is defined by \eqref{eq_ouralpha} and, without loss of generality, we take $x_0 = 0$. We begin by establishing a lemma concerning the improvement of the flatness of solutions.

\begin{lemma}[{\bf Improvement of flatness}]\label{lem:stepone}
Let $u \in C^0(B_1)$ be a solution to \eqref{P} with $0 \in \mathrm{C}_{\mathcal{Z}}(u)$. Suppose that 
there exists $\eps_0>0$ such that  one of the following conditions is satisfied
\begin{itemize}
\begin{multicols}{2}
\item[(i)] $\dinf_{B_1} u \geq -\eps_0,$
\item[(ii)] $\dfrac{|\{u<0\}\cap B_1|}{|B_1|} \leq \eps_0$.
\end{multicols} 
\end{itemize}
Then, given $\delta\in (0,1)$, there exists $ \varepsilon(\delta)>0$ such that 
$
\|f\|_{L^{\infty}({B_1})} \leq \varepsilon,
$ implies
\[
\sup_{B_{\frac{1}{2}}}|u| \leq 1 - \delta.
\]
\end{lemma}

\begin{proof}
Suppose by contradiction that there exists $\delta_0\in (0,1)$ and sequences $(u_k)_{k \in \mathbb{N}}$, $(f_k)_{k \in \mathbb{N}}$ such that $\|u_k\|_{L^\infty({B_1)}} \leq  1$, $0 \in \mathrm{C}_{\mathcal{Z}}(u_k)$, 
\[
\inf_{B_1} u_k \geq - \frac{1}{k} \quad \mbox{or} \quad \dfrac{|\{u_k<0\}\cap B_1|}{|B_1|} \leq 1/k,
\]
also $u_k$ solves in the viscosity sense
\begin{equation}\label{eq_cont1}
\Delta_\infty u_k = \mathcal{G}(x, u_k, Du_k) \lesssim |f_k(x)||u_k|^m\min\{1, |D u_k|^\k\} \quad \mbox{in} \quad  B_1,    
\end{equation}
but
\begin{align}\label{sup}
\sup_{B_{\frac{1}{2}}}|u_k|>1 - \delta_0.
\end{align}
From the regularity available for \eqref{eq_cont1}, we infer that $(u_k)_{k \in \mathbb{N}}$ is uniformly bounded in $C^0(\bar{B}_{1/2})$. Hence, there exists a function $u_\infty \in C^{0, \beta}(B_{1/2})$ (Theorem \ref{Reg_Lipsc}), for some $\beta \in (0,1)$, such that $u_k \to u_\infty$ uniformly in $\bar{B}_{1/2}$. Furthermore, $\|u_\infty\|_{L^\infty({B_{1/2}})}\leq  1$, $u_\infty(0)=0$, $u_\infty$ satisfies
\[\Delta_\infty u_\infty=0 \quad \mbox{in} \quad  B_\frac{1}{2} \quad \mbox{with} \quad  \dsup_{B_{\frac{1}{2}}}|u_\infty|\geq1 - \delta_0,
\]
and also 
\[
\inf_{B_{1/2}} u_\infty \geq 0 \quad \mbox{ or } \quad |\{u_\infty<0\}\cap B_{1/2}| = 0.
\]
Since $u_\infty\in C^0(\bar{B}_{1/2})$, we have $u_\infty \geq 0$, which along with \eqref{sup} contradicts the Maximum Principle (Lemma \ref{MP}).
\end{proof}

\begin{lemma}\label{lem:steptwo} Let $u \in C^0(B_1)$ be a normalized viscosity solution to \eqref{P} with $0 \in \mathrm{C}_{\mathcal{Z}}(u)$. If 
\[
\|f\|_{L^\infty(B_1)} \leq \eps
\]
at least one of the following conditions
\begin{itemize}
\begin{multicols}{2}
\item[(i)] $\dinf_{B_1} u \geq -\eps$,
\item[(ii)] $\dfrac{|\{u<0\}\cap B_1|}{|B_1|} \leq \eps$,
\end{multicols} \end{itemize}
holds true, then
\[
\dsup_{B_{{2^{-k}}}}|u| \leq \dfrac{1}{2^{k\alpha}},
\]
for every $k \in \mathbb{N}$ and $\alpha>0$ as in \eqref{eq_ouralpha}.
\end{lemma}

\begin{proof}
We argue by induction on $k$. Taking $\delta = 1 - 2^{-\alpha}$, which determines $\eps :=\eps (\delta)$ through Lemma \ref{lem:stepone}, we obtain
\[
\sup_{B_{{2^{-1}}}}|u| \leq \dfrac{1}{2^{\alpha}},
\]
which means the statement for $k=1$. Now, suppose that the case $k = 1, \ldots, n$ is already verified, and let us prove the case $k= n+1$. Define the auxiliary function $v : \mathbb{R}^d \to \mathbb{R}$
\[
v_k(x):= \dfrac{u(2^{-k}x)}{2^{k\alpha}}.
\]
We have that $\|v_k\|_{L^\infty({B_1})} \leq 1$, $0 \in \mathrm{C}_{\mathcal{Z}}(v_k)$ and $v_k$ satisfies at least one of the conditions $(i)-(ii)$. Indeed, notice that
\[
\inf_{B_1}v_k = 2^{-k\alpha}\inf_{B_{2^{-k}}}u \geq 2^{-k\alpha}\dinf_{B_1} u \geq -2^{-k\alpha}\eps > -\eps,
\]
or
\[
\dfrac{|\{v_k<0\}\cap B_1|}{|B_1|} =  \dfrac{|\{u<0\}\cap B_{2^{-k}} |}{|B_1|} \leq \dfrac{|\{u<0\}\cap B_1|}{|B_1|} \leq \eps.
\]
Moreover, setting $r = 2^{-k}$, $v_k$ solves in the viscosity sense
\begin{align*}
	\Delta_\infty v_k(x) & = \langle D^2 v_k(x) D v_k(x), D v_k(x)\rangle =  r^{4-3\alpha}(\Delta_\infty u)(rx)\\ 
    & =  r^{4-3\alpha} \mathcal{G}(rx, u(rx), (Du)(rx)) = r^{4-3\alpha} \mathcal{G}(rx, r^{\alpha}v_k(x), r^{\alpha-1}Dv_k(x))\\
    & \lesssim r^{4-3\alpha} r^{\alpha m} r^{(\alpha -1)\kappa}f(rx)|v_k|^m \min\{1, |Dv_k|^{\kappa} \}\\
  %  & = f(rx)|u(rx)|^m|Du(rx)|^\k r^{4 - 3 \alpha}|v_k|^m \min\{1, |Dv_k|^{\kappa} \} \\
%    & = f_k(x) |v_k|^m r^{\alpha m} |D v_k|^\k r^{(\alpha -1)\k} r^{4-3\alpha}\\ 
    & =  f_k(x) |v_k|^m \min\{1, |Dv_k|^{\kappa} \} r^{4 - \k - \alpha(3 - m -\k)} \\
    & =  f_k(x) |v_k|^m| \min\{1, |Dv_k|^{\kappa} \},
\end{align*}
where $f_k(x) := f(2^{-k}x)$. Thus, applying Lemma \ref{lem:stepone} to $v_k$ we obtain
\[
\dsup_{B_{\frac{1}{2}}}|v| \leq \dfrac{1}{2^\alpha},
\]
and rescaling back to $u$ we conclude
\[
\dsup_{B_{{2^{-(k+1)}}}}|u(x)| \leq \dfrac{1}{2^{(k+1)\alpha}}.
\]	
\end{proof}

\begin{proof}[{\bf Proof of Theorem \ref{maintheo_1}}]
Given $r\in (0,1]$, we take $k\in \n$
%and $r_0 \in (r,1]$
such that $2^{-(k+1)} < r\leq 2^{-k}$. By applying Lemma \ref{lem:steptwo} we can infer that
\begin{align*}
\sup_{B_{r}}|u(x)| & \leq \sup_{B_{{2^{-k}}}}|u(x)|  \leq \dfrac{1}{2^{k\alpha}}  \leq \mathrm{C}r^\alpha.
%\leq C|x|^\alpha, \quad \mbox{for} \quad C \geq \dfrac{r^{\alpha}_0}{|x|^\alpha}.
\end{align*}
\end{proof}

\section{Regularity estimates in specific dimensions} \label{S5}

Throughout this section, we assume that the viscosity solutions to
\[
\Delta_\infty \mathfrak{h} = 0 \;\;\mbox{ in }\; B_1 \subset \mathbb{R}^n,
\]
are of class $C_{loc}^{1, \alpha^{\mathrm{H}}_n}$. Moreover, remember the following definition of the exponent (see, \eqref{eq_alpha2})
\begin{equation*}
\hat{\alpha} \in  \left(0, \alpha^{\mathrm{H}}_n\right) \cap \left(0, \frac{1+m}{3-(m+\kappa)}\right],
\end{equation*}
where $\alpha^{\mathrm{H}}_n \in (0,1)$ is defined by \eqref{eq_alphan}.

\begin{lemma}[{\bf Approximation Lemma}]\label{lem_app}
Let $u \in C^0(B_1)$ be a normalized viscosity solution to \eqref{P} with $x_0 \in \mathrm{C}_{\mathcal{Z}}(u)$. Given $\delta > 0$, there exists $\varepsilon > 0$ such that if
\[
\|f\|_{L^\infty(B_1)} \leq \varepsilon,
\]
then we can find $\mathfrak{h} \in \Xi_n^{\mathrm{H}}(B_1)$, with $x_0 \in \mathrm{C}_{\mathcal{Z}}(\mathfrak{h})$ satisfying
\begin{equation}\label{eq_20}
\sup_{B_{3/4}}|u(x)-\mathfrak{h}(x)| \leq \delta.
\end{equation}
\end{lemma}

\begin{proof}
We argue by contradiction. Suppose that the statement is false, then there exists $\delta_0$ and sequences $(u_k)_{k\in\mathbb{N}}$, $(f_k)_{k\in\mathbb{N}}$, satisfying
\begin{equation}\label{eq_22}
\Delta_\infty u_k=\mathcal{G}(x, u_k, Du_k) \lesssim |f_k(x)||u_k|^m\min\{1, |D u_k|^\k\} \quad \mbox{in} \quad  B_1,    
\end{equation}
\begin{equation}\label{eq_23}
\|f_k\|_{L^\infty(B_1)} \leq \frac{1}{k},
\end{equation}
\begin{equation}\label{eq_24}
x_0 \in \mathrm{C}_{\mathcal{Z}}(u_k), \mbox{ for every k}\in \mathbb{N},
\end{equation}
but
\begin{equation}\label{eq_21}
\sup_{B_1}|u_k(x)-\mathfrak{h}(x)| > \delta_0,
\end{equation}
for every $\mathfrak{h} \in \Xi_n^{\mathrm{H}}(B_1)$.
	
From the regularity available for viscosity solutions to \eqref{eq_22}, we can ensure that the family $(u_k)_{k \in \mathbb{N}}$ is locally of class $C_{loc}^{0,\beta}$, for some $\beta \in (0,1)$. Hence, there exists a function $u_\infty$ such that, up to a subsequence, $u_k \to u_\infty$  locally uniformly in $B_1$ and $u_\infty(x_0)=0$. By using the stability of viscosity solutions, \eqref{eq_22} and \eqref{eq_23}, we have that $u_\infty$ solves
\[
\Delta u_\infty = 0 \;\;\mbox{ in }\; B_1,
\]
which in particular implies that, we have that $u_\infty \in C_{loc}^{1, \alpha_n^{\mathrm{H}}}(B_1)$. Now, notice that since $u_k(x_0) = u_\infty(x_0) = 0$, we have
\begin{align*}
\partial_{x_i}u_\infty(x_0) & = \lim_{h \to 0} \frac{u_\infty(x_0+he_i) - u_\infty(x_0)}{h} \\
& = \lim_{h \to 0} \left[ \frac{u_\infty(x_0+he_i) \pm u_k(x_0+he_i) \pm u_k(x_0) - u_\infty(x_0)}{h} \right] \\
& = \lim_{h \to 0} \left[ \frac{u_\infty(x_0+he_i) - u_k(x_0+he_i)}{h} + \frac{u_k(x_0+he_i)-u_k(x_0)}{h}\right] \\
& = \lim_{h \to 0} \frac{u_\infty(x_0+he_i) - u_k(x_0+he_i)}{h} + \lim_{h \to 0}\frac{u_k(x_0+he_i)-u_k(x_0)}{h} \\
& = \lim_{h \to 0} \frac{u_\infty(x_0+he_i) - u_k(x_0+he_i)}{h}.
\end{align*}
Notice that fixed $k$, the quantity above converges point-wisely to $\partial_{x_i}u_\infty(x_0)$ as $h \to 0$. Moreover, it also converges uniformly for fixed $h\not= 0$, as $1/k \to 0$. Hence, by using the Moore-Osgood theorem, we obtain
\begin{align*}
\partial_{x_i}u_\infty(x_0) & = \lim_{k \to \infty}\lim_{h \to 0} \frac{u_\infty(x_0+he_i) - u_k(x_0+he_i)}{h} = \lim_{h \to 0}\lim_{k \to \infty}\frac{u_\infty(x_0+he_i) - u_k(x_0+he_i)}{h} = 0.
\end{align*}
Therefore, we can conclude that $Du_\infty(x_0)=0$, and thus $x_0 \in\mathrm{C}_{\mathcal{Z}}(u_\infty)$. Finally, by taking $\mathfrak{h} \equiv u_\infty$, we reach a contradiction with \eqref{eq_21} for $k$ sufficiently large.
\end{proof}

\begin{lemma}\label{lem_step1}
Let $u \in C^0(B_1)$ be a viscosity solution to \eqref{P} with $x_0 \in \mathrm{C}_{\mathcal{Z}}(u)$. Then, there exists $\varepsilon > 0$, such that if
\[
\|f\|_{L^\infty(B_1)} \leq \varepsilon,
\]
then we can find $0< \rho \ll 1$, satisfying
\[
\sup_{B_\rho(x_0)}|u(x)| \leq \rho^{1+\hat\alpha}.
\]
\end{lemma}

\begin{proof}
Fix $\delta > 0$ to be determined later, and let $\mathfrak{h} \in \Xi_n^{\mathrm{H}}(B_1)$ be the function from Lemma \ref{lem_app}. Since $x_0 \in \mathrm{C}_{\mathcal{Z}}(\mathfrak{h})$, we can infer from the regularity of $h$ that for $\rho > 0$ sufficiently small we have
\[
\sup_{B_\rho(x_0)}|h(x)| \leq \mathrm{C}\rho^{1+\alpha_n^{\mathrm{H}}}.
\]
Now, we estimate
\begin{align*}
\sup_{B_\rho(x_0)}|u(x)| & = \sup_{B_\rho(x_0)}|u(x)-h(x)| + \sup_{B_\rho(x_0)}|h(x)|  \leq \delta + \mathrm{C}\rho^{1+\alpha_n^{\mathrm{H}}}.
\end{align*}
We make the universal choices,
\[
\rho = \left(\frac{1}{2\mathrm{C}}\right)^{\frac{1}{\alpha_n^{\mathrm{H}}-\hat\alpha}} \;\;\; \mbox{ and }\;\; \delta=\frac{\rho^{1+\hat\alpha}}{2},
\]
to conclude that
\[
\sup_{B_\rho(x_0)}|u(x)| \leq \rho^{1+\hat\alpha}.
\]
Notice that the choice of $\delta$ determines the value of $\varepsilon$ via Lemma \ref{lem_app} and $\hat{\alpha}>0$ is as .
\end{proof}

\begin{lemma}\label{lem_step2}
Let $u \in C^0(B_1)$ be a viscosity solution to \eqref{P} with $x_0 \in \mathrm{C}_{\mathcal{Z}}(u)$. There exists $\varepsilon > 0$, such that if
\[
\|f\|_{L^\infty(B_1)} \leq \varepsilon,
\]
then
\begin{equation}\label{eq_25}
\sup_{B_{\rho^k}(x_0)}|u(x)| \leq \rho^{k(1+\hat\alpha)},
\end{equation}
where $\rho \in (0, 1)$ comes from Lemma \ref{lem_step1}.
\end{lemma}

\begin{proof}
We resort to an induction argument. The case $k=1$ follows directly from Lemma \ref{lem_step1}. Suppose that we have verified the statement for the cases $j=1,\ldots,k$. We shall prove the case $j=k+1$. We introduce the auxiliary function
\[
v_k(x) = \frac{u(\rho^kx)}{\rho^{k(1+\hat\alpha)}}.
\]
From \eqref{eq_25}, we have that $\|v_k\|_{L^\infty(B_1)} \leq 1$. Moreover, $v_k$ solves
\[
\Delta_\infty v_k(x) \lesssim  f(\rho^kx) |v_k|^m \min\{1, |Dv_k|^{\kappa} \} \rho^{k[1-3\hat\alpha +m(1+\hat\alpha) + \kappa\hat\alpha]}\\
    = f_k(x) |v_k|^m \min\{1, |Dv_k|^{\kappa} \},
\]
where $\tilde{f}(x) := f(\rho^kx)$. The choice of $\hat\alpha$ in \eqref{eq_alpha2}, ensures that $v_k$ also satisfies the hypotheses of Lemma \ref{lem_step2}, hence
\[
\sup_{B_{\rho}(x_0)}|v_k(x)| \leq \rho^{1+\hat\alpha},
\]
which yields to
\[
\sup_{B_{\rho^{k+1}}(x_0)}|u(x)| \leq \rho^{(k+1)(1+\hat\alpha)}.
\]
\end{proof}

\begin{proof}[{\bf Proof of Theorem \ref{maintheo_2}}]
Let $0 < r \ll 1/2$ and $x_0 \in \mathrm{C}_{\mathcal{Z}}(u)$. Consider $k \in \mathbb{N}$ such that $\rho^{k+1} \leq r \leq \rho^k$. We estimate
\begin{align*}
\sup_{B_r(x_0)}|u(x)| & \leq  \sup_{B_{\rho^k}(x_0)}|u(x)| \leq \rho^{k(1+\hat\alpha)} \leq \rho^{-(1+\hat\alpha)}\rho^{(k+1)(1+\hat\alpha)}  \leq \mathrm{C}r^{1+\hat\alpha}.
\end{align*}
\end{proof}

\begin{remark}[{\bf Normalization and smallness regime}] We observe that, by considering $v(x)= \tfrac{u(\mathrm{R}x)}{\mathcal{N}}$, for some $\mathrm{R}, \mathcal{N} >0$ to be chosen \textit{a posteriori}, $x_0 \in \mathrm{C}_{\mathcal{Z}}(u)\cap B_{1/2}$, and $x_{\mathrm{R}} = x_0 + \mathrm{R}x$ we have
\begin{align*}
\Delta_\infty v(x) &= \frac{\mathrm{R}^4}{\mathcal{N}^3} (\Delta_\infty u )(x_{\mathrm{R}})  = \frac{\mathrm{R}^4}{\mathcal{N}^3}  \cG(x_{\mathrm{R}}, u(x_{\mathrm{R}}), (Du)(x_{\mathrm{R}})) = \frac{\mathrm{R}^4}{\mathcal{N}^3} \cG(x_{\mathrm{R}}, \mathcal{N}v,\tfrac{\mathcal{N}}{\mathrm{R}} Dv) \\
& \lesssim \mathrm{R}^{4-\kappa} \mathcal{N}^{m+\kappa-3}|f(x_{\mathrm{R}})||v|^m\min\{1,|Dv|^\kappa \}\\
& =: f_{\mathrm{R}}(x)|v|^m\min\{1,|Dv|^\kappa \}. \end{align*}

Therefore, given $\varepsilon>0$, taking 
$$
\mathcal{N} = \|u\|_{L^\infty(B_1)} \quad  \text{and} \quad  0<\mathrm{R}< \min\left\{\frac{1}{4}\mathrm{dist}(\mathrm{C}_{\mathcal{Z}}(u)\cap B_{1/2}, \partial B_1), \omega^{-1}_f(\varepsilon),  \|u\|_{L^\infty(B_1)}^{\frac{3-(m+\kappa)}{4-\kappa}}\right\},
$$
where $\omega_f: [0, \infty) \to [0, \infty)$ is a modulus of continuity for the weight function $f$, it follows that  
\[
\|f_{\mathrm{R}}\|_{L^\infty(B_1)} = \mathrm{R}^{4-\kappa} \mathcal{N}^{m+\kappa-3}\|f\|_{L^\infty(B_1)} < \varepsilon \quad \text{and} \quad \|v\|_{L^{\infty}(B_1)}\leq 1.
\]
In other words, the normalization assumption for the solution and smallness regime on weight function are not restrictive in obtaining our regularity estimates. 
\end{remark}

\section{Reflecting estimates: Proof of Theorem \ref{ThmPositive-Part}} \label{S6}

In this final section, based on some ideas from \cite{BF24}, we will deal with the proof of Theorem \ref{ThmPositive-Part}. 

\begin{proof}[{\bf Proof of Theorem \ref{ThmPositive-Part}}]
We only consider the case for $u^-$, since the other case is analogous. By following the ideas from \cite{CKS00}, we first prove that there exists $\tilde{\mathrm{C}}_1 > 0$, such that
\begin{equation}\label{eq_flip}
\mathfrak{s}(k+1) \leq \max\left\{\dfrac{\tilde{\mathrm{C}}_1}{2^{\alpha(k+1)}}, \dfrac{1}{2^\alpha}\mathfrak{s}(k)\right\},    
\end{equation}
where $\mathfrak{s}(k_j) := \|u\|_{L^\infty(B_{2^{-k_j}}(x_0))}$. We argue by contradiction. Suppose that 
\begin{equation}\label{eq_contposi}
\mathfrak{s}(k_j+1) > \max\left\{\frac{j}{2^{\alpha(k_j+1)}}, \frac{1}{2^\alpha}\mathfrak{s}(k_j)\right\},
\end{equation}
Now, consider the rescaled function
\[
v_j(x) = \dfrac{u(x_0+2^{-k_j}x)}{\mathfrak{s}(k_j+1)}.
\]
 Notice that $v_j(0)= 0 $ for all $j \in \mathbb{N}$, and it solves in the viscosity sense
\[
\Delta_\infty v_j(x) \lesssim 2^{(\kappa - 4)k_j}\mathfrak{s}(k_j+1)^{m+\kappa-3}\tilde{f}(x)|v_j(x)|^m\min\left\{1, |Dv_j(x)|^{\kappa}\right\} =: g_j(x)
\]
where $\tilde{f}(x) = f(x_0+2^{-k_j}x)$. From \eqref{eq_contposi}, and the explicit value of $\alpha$ in \eqref{eq_ouralpha}, we obtain 
\[
2^{(\kappa - 4)k_j}\mathfrak{s}(k_j+1)^{m+\kappa-3} \leq \frac{2^{(\kappa - 4)k_j}2^{\alpha(k_j + 1)[3-(m+\kappa)]}}{j^{3-(m+\kappa)}} = \frac{2^{(4-\kappa )}}{j^{3-(m+\kappa)}}.
\]
Hence,
\[
|g_j(x)| \leq \frac{2^{(4-\kappa )}}{j^{3-(m+\kappa)}}\|f\|_{L^\infty(B_1)}, 
\]
 for any $j\in \mathbb{N}$. Since $g$ is bounded, standard regularity and stability results (see, Theorem \ref{Reg_Lipsc} and Lemma \ref{stability}), assure the existence of $v_\infty$ such that $v_j \to v_\infty$ locally uniformly,  as $j\to \infty$, and
\begin{equation}\label{eq_flip1}
\Delta_\infty v_\infty = 0,  \quad \mbox{in } B_{{3}/{4}}   
\end{equation}
in the viscosity sense, see Lemma \ref{stability} and Theorem \ref{Reg_Lipsc}. By using \eqref{eq_contposi} once again we have
\[
v_j^- = \dfrac{u^-(x_0+2^{-k_j}x)}{\mathfrak{s}(k_j+1)} \leq \dfrac{2^{-\alpha k_j}\mathrm{C}_0}{\mathfrak{s}(k_j+1)} \leq \dfrac{2^{-\alpha k_j}\mathrm{C}_02^{\alpha(k_j+1)}}{j} \leq \dfrac{2^{-\alpha}\mathrm{C}_0}{j} \to 0,
\]
as $j \to \infty$, which implies
\begin{equation}\label{eq_flip2}
v_\infty \geq 0 \quad \mbox{ in } \; B_{3/4}.    
\end{equation}
Moreover
\[
\sup_{B_1}|v_j| \leq 2^\alpha \quad \mbox{ and } \quad \sup_{B_{1/2}}v_j = 1,
\]
which leads to 
\begin{equation}\label{eq_flip3}
\sup_{B_1}|v_\infty| \leq 2^\alpha \quad \mbox{ and } \quad \sup_{B_{1/2}}v_\infty = 1,   
\end{equation}

From \eqref{eq_flip1}, \eqref{eq_flip2} and \eqref{eq_flip3} we get a contradiction from the strong maximum principle, and then \eqref{eq_flip} holds. To finish the proof, let $k_0$ be the smallest integer such that $2^{-\alpha k} \leq r_0$ and set $\mathrm{C} = \max\{\tilde{\mathrm{C}}_1, 2^{\alpha k_0}\mathfrak{s}(k_0)\}$. From this, we can show that
\[
\mathfrak{s}(k) \leq \mathrm{C}2^{-\alpha k}.
\]
Finally, take $r \in (0,r_0]$ and $k \geq k_0$ such that $2^{-\alpha(k+1)} \leq r \leq 2^{-\alpha k}$ and observe that
\begin{align*}
\|u\|_{L^\infty(B_r(x_0))}  \leq & \|u\|_{L^\infty(B_{2^{-k}}(x_0))}
 =  \mathfrak{s}(k) 
 \leq \mathrm{C}2^{-\alpha k}
  \leq \mathrm{C}r^{\alpha},
\end{align*}
which implies
\begin{align*}
\sup_{B_r(x_0)}u^+ &\leq  \sup_{B_r(x_0)}u + \sup_{B_r(x_0)}u^-\leq \mathrm{C}r^\alpha + \mathrm{C}_0r^\alpha  = \mathrm{C}_1r^\alpha.
\end{align*}
This finishes the proof.
\end{proof}

\section{Non-degeneracy of solutions: Proof of Theorem \ref{ThmNon-Degen}}
\label{S7}
In this section, we will address the non-degeneracy properties of solutions.

\begin{proof}[{\bf Proof of Theorem \ref{ThmNon-Degen}}] Firstly, consider the auxiliary comparison function
$$
v(x) := \mathrm{c}_0|x-x_0|^{\alpha} +u(x_0) \quad \text{in} \quad B_r(x_0).
$$
for a universal constant $\mathrm{c}_0>0$ to be chosen \textit{a posteriori}.

Straightforward computation yields to
$$
Dv(x)  = \alpha\mathrm{c}_0|x-x_0|^{\alpha-1}\frac{(x-x_0)}{|x-x_0|}.
$$
and
$$
\Delta_{\infty} v(x) = (\alpha\mathrm{c}_0)^3(\alpha-1)|x-x_0|^{3\alpha-4} \quad \text{in} \quad B_1
$$
in the viscosity sense. Hence, by choosing
$$
\mathrm{c}_0 = \sqrt[3]{\frac{\theta}{\alpha^3(\alpha-1)}}
$$
it holds in the viscosity sense
$$
\begin{array}{rcl}
  \Delta_{\infty} v(x) & \le & \theta r^{3\alpha-4}  <  \displaystyle \inf_{B_r(x_0)}\mathcal{G}(x, u, Du)   \le \mathcal{G}(x, u, Du) \le \Delta_{\infty} u(x) \quad \text{in}  \quad B_r(x_0).
\end{array}
$$

Therefore, from the Comparison Principle (see, Lemma \ref{LemmaCP}), since $v(x_0) = u(x_0)$, for each $r > 0$, there must exist a
$x_r \in \partial B_r(x_0)$ such that
$$
v(x_r)\le u(x_r),
$$
which implies the desired estimate in \eqref{EqNDeg}.
\end{proof}

%\section{Appendix: Existence of viscosity solutions}

\bigskip
\bigskip

\noindent{\bf Acknowledgment:}
J.V. da Silva has received partial support from CNPq-Brazil under Grant No. 307131/2022-0 and FAEPEX-UNICAMP 2441/23 (Editais Especiais - PIND - Projetos Individuais, 03/2023), and Chamada CNPq/MCTI No. 10/2023 - Faixa B - Consolidated Research Groups under Grant No. 420014/2023-3. This project began at Unicamp's Summer School in 2024. Therefore, the authors would like to thank Unicamp's Department of Mathematics for providing a productive and stimulating scientific atmosphere during the visit of M. Santos and M. Soares. J.V. da Silva also expresses his gratitude to the Department of Mathematics at UnB for the warm reception during his visit to the XVII Summer Workshop in Mathematics 2025, which has positively contributed to the outcome of this work.  M. Santos is partially supported by the Portuguese government through FCT-Funda\c c\~ao para a Ci\^encia e a Tecnologia, I.P., under the projects UID/04561/2025 and the Scientific Employment Stimulus - Individual Call (CEEC Individual), DOI identifier:\\
 https://doi.org/10.54499/2023.08772.CEECIND/CP2831/CT0003

\bigskip
\noindent\textsc{Jo\~{a}o Vitor da Silva}\\
Departamento de Matem\'{a}tica\\
Instituto de Mamatem\'{a}tica, Estat\'{i}stica e Computa\c{c}\~{a}o Cient\'{i}fica - IMECC\\
Universidade Estadual de Campinas - Unicamp\\
 Rua S\'{e}rgio Buarque de Holanda, 651, CEP 13083-859, Campinas, SP, Brazil\\
\noindent\texttt{jdasilva@unicamp.br}
\bigskip

\noindent\textsc{Makson S. Santos}\\
Center for Mathematical Studies\\
University of Lisbon (CEMS.UL)\\
1749-016 Lisboa, Portugal\\
\noindent\texttt{msasantos@fc.ul.pt}
\bigskip

\noindent\textsc{Mayra Soares}\\
Departamento de Matem\'{a}tica,\\ Universidade de Bras\'{i}lia - UnB\\
Instituto Central de Ci\^{e}ncias, Campus Darci Ribeiro,\\
70910-900, Asa Norte, Bras\'{i}lia, Distrito Federal, Brasil\\
\noindent\texttt{mayra.soares@unb.br}


\begin{thebibliography}{99}

\bibitem{ALT16} Ara\'{u}jo, D.J., Leit\~{a}o, R. and Teixeira, E.
\textit{Infinity Laplacian equation with strong absorptions}. J. Funct. Anal. 270 (2016), no.6, 2249–2267.

\bibitem{ArmSmart10} Armstrong, S.N. and Smart, C.K. 
\textit{An easy proof of Jensen's theorem on the uniqueness of infinity harmonic functions}.
Calc. Var. Partial Differential Equations 37 (2010), no.3-4, 381–384.

\bibitem{Aronsson67}  Aronsson, G. 
\textit{ Extension of functions satisfying Lipschitz conditions}, Ark. Mat. 6 (1967) 551-561.

\bibitem{Aronsson68} Aronsson, G. 
\textit{On the partial differential equation $u^2_xu_{xx} + 2u_xu_yu_{xy} + u^2_yu_{yy} = 0$}, Arkiv f\"{o}r Matematik 7 (1968), pp. 395–435.

\bibitem{Aronsson84}  Aronsson, G. 
\textit{On certain singular solutions of the partial differential
equation $u^2_xu_{xx} + 2u_xu_yu_{xy} + u^2
_yu_{yy} = 0$}, Manuscripta Mathematica 47 (1984), pp. 133-151.

%\bibitem{Arons89} Aronsson, G.
%\textit{Representation of a $p$-harmonic function near a critical point in the plane}.
%Manuscripta Math. 66 (1989), no.1, 73-95.

\bibitem{ACJ04} Aronson, G., Crandall, M.G. and  Juutinen, P.
\textit{A tour of the theory of absolute minimizing
functions}. Bull. Amer. Math. Soc. 41 (2004), 439–505.

\bibitem{BEJ08} Barron, E.N.,  Evans, L.C. and  Jensen, R.
\textit{The infinity Laplacian, Aronsson’s equation and their generalizations}.
Trans. Amer. Math. Soc., 360 (1) :77-101, 2008.

\bibitem{BFH86}  Batt, J., Faltenbacher, W. and   Horst, E.
\textit{Stationary spherically symmetric models in stellar
dynamics}.
Arch. Rational Mech. Anal., 93 (2) :159-183, 1986.

\bibitem{BF24} Bayrami, M. and  Fotouhi, M.
\textit{Regularity in the two-phase Bernoulli problem for the p-Laplace operator}.
Calc. Var. Partial Differential Equations 63 (2024), no. 7, Paper No. 183, 38 pp.

\bibitem{BeJDaSNS2024} Bezerra J\'{u}nior, E.C., da Silva, J.V., Nascimento, T. and S\'{a}, G.S.
\textit{Improved regularity estimates for  H\'{e}non-type equations driven by the $\infty$-Laplacian}. 2024 Arxiv \url{
https://doi.org/10.48550/arXiv.2410.19970}.

\bibitem{BhatMoh11} Bhattacharya, T. and  Mohammed, A.
\textit{On solutions to Dirichlet problems involving the infinity-Laplacian}. Adv. Calc. Var. 4 (2011), no.4, 445–487.

\bibitem{BhatMoh12} Bhattacharya, T. and  Mohammed, A.
\textit{Inhomogeneous Dirichlet problems involving the infinity-Laplacian}. Adv. Differential Equations 17 (2012), no.3-4, 225–266.

\bibitem{BisVo23} Biswas, A. and Vo, H.-H.
\textit{Liouville theorems for infinity Laplacian with gradient and KPP type equation}. Ann. Sc. Norm. Super. Pisa Cl. Sci. (5) 24 (2023), no.3, 1223–1256.

\bibitem{CKS00} Caffarelli, L.A., Karp, L. and  Shahgholian, H. 
\textit{Regularity of a free boundary with application to the Pompeiu problem}. Ann. of Math. (2) 151 (2000), no.1, 269–292.

%\bibitem{Cald-Zyg61} Calder\'{o}n, A.-P.and Zygmund, A.
%\textit{Local properties of solutions of elliptic partial differential equations}.
%Studia Math. 20 (1961), 171-225.


%\bibitem{CrastFrag20} Crasta, G. and Fragal\`{a}, I.
%\textit{Bernoulli free boundary problem for the infinity Laplacian}. SIAM J. Math. Anal. 52 (2020), no.1, 821–844.

%\bibitem{DS04} Damascelli, L. and Sciunzi, B.
%\textit{Regularity, monotonicity and symmetry of positive solutions of $m$-Laplace equations}.
%J. Differential Equations, 206(2), 2004, pp. 483-515

\bibitem{daSdosPRS} da Silva, J.V., dos Prazeres, D.S., Ricarte, G.C. and  S\'{a}, G.S.  \textit{Sharp and improved regularity estimates for weighted quasilinear elliptic equations of $p-$Laplacian type and applications}, 2024 Arxiv \url{https://doi.org/10.48550/arXiv.2410.14862}.

\bibitem{daSRosSal19} da Silva, J.V., Rossi, J.D. and Salort, A.M.
\textit{Regularity properties for $p$-dead core problems and their asymptotic limit as $p \to \infty$}.
J. Lond. Math. Soc. (2) 99 (2019), no.1, 69-96.

\bibitem{daSSal18} da Silva, J.V. and Salort, A.M.
\textit{Sharp regularity estimates for quasi-linear elliptic dead core problems and applications}.
Calc. Var. Partial Differential Equations 57 (2018), no.3, Paper No. 83, 24 pp.

\bibitem{Diaz85} Díaz, J. I.
\textit{Nonlinear partial differential equations and free boundaries}. Vol. I.
Elliptic equations. Res. Notes in Math., 106 Pitman (Advanced Publishing Program), Boston, MA, 1985. vii+323 pp. ISBN:0-273-08572-7.

\bibitem{DT20} Diehl, N.M.L. and Teymurazyan, R.,
\textit{Reaction-diffusion equations for the infinity Laplacian}. Nonlinear Anal. 199 (2020), 111956, 12 pp.

\bibitem{ES08} Evans, L.C. and  Savin, O.
\textit{$C^{1,\alpha}$ regularity for infinity harmonic functions in two dimensions}. Calc. Var. Partial Differential Equations 32 (2008), no.3, 325–347.

\bibitem{ESmart11} Evans, L.C. and Smart, C.K.
\textit{Everywhere differentiability of infinity harmonic functions}. Calc. Var. Partial Differential Equations 42 (2011), no.1-2, 289–299.

%\bibitem{gt} Gilbarg, D. and Trudinger, N. S.
%\textit{Elliptic partial differential equations of second order}.
%Reprint of the 1998 edition. Classics in Mathematics, Springer-Verlag, Berlin, 2001.

%\bibitem{IwaManf89} Iwaniec, T. and Manfredi, J.J.
%\textit{Regularity of $p$-harmonic functions on the plane}.
%Rev. Mat. Iberoamericana 5 (1989), no. 1-2, 1-19.

\bibitem{Jensen93} Jensen, R.
\textit{Uniqueness of Lipschitz extensions: minimizing the sup norm of the gradient}. Arch. Rational Mech. Anal. 123 (1993), no.1, 51–74.

%\bibitem{LiLiu23} Li, C. and Liu, F.
%\textit{Viscosity solutions to the infinity Laplacian equation with lower terms}. Electron. J. Differential Equations (2023), Paper No. 42, 23 pp.

%\bibitem{LY12} Liu, F. and Yang, X.-P.
%\textit{Solutions to an inhomogeneous equation involving infinity Laplacian}. Nonlinear Anal. 75 (2012), no.14, 5693–5701.

\bibitem{Lind14} Lindgren, E.
\textit{On the regularity of solutions of the inhomogeneous infinity Laplace equation}.
Proc. Amer. Math. Soc. 142 (2014), no.1, 277–288.

\bibitem{LuWang08} Lu, G. and Wang, P.
 \textit{Inhomogeneous infinity Laplace equation}. Adv. Math. 217 (2008), no.4, 1838–1868.

%\bibitem{QhanFlin} Han, Q. and Faghua, L.,
%\textit{Elliptic Partial Differential Equations}.
%Courant Institute of Mathematical Sciences,
%New York University, Providence, R.I.: AMS, New York (2011).


%\bibitem{Leoni} Leoni, G.
%\textit{A first course in Sobolev spaces}.
%Second edition
%Grad. Stud. Math., 181
%American Mathematical Society, Providence, RI, 2017. xxii+734 pp.
%ISBN:978-1-4704-2921-8.


%\bibitem{LL}  Lindgren, E. and Lindqvist, P.,
%\textit{Regularity of the $p$-Poisson equation in the plane}.
%J. Anal. Math. 132 (2017), 217-228.

%\bibitem{Lou08} Lou, H.
%\textit{On singular sets of local solutions to $p$-Laplace equations}.
%Chinese Ann. Math. Ser. B29 (2008), no.5, 521-530.

%\bibitem{MZ97} Mal\'{y}, J. and Ziemer, W.P.,
%\textit{Fine Regularity of Solutions of Elliptic Partial Differential Equations},
% Mathematical Surveys and Monographs, 51.
% American Mathematical Society, Providence, RI, 1997. xiv+291 pp. ISBN: 0-8218-0335-2.

%\bibitem{Manf88} Manfredi, J.J.
%\textit{$p$-harmonic functions in the plane}.
%Proc. Amer. Math. Soc. 103 (1988), no.2, 473-479.

\bibitem{NSST23} Nascimento, T.M., S\'{a}, G., Sobral, A. and Teixeira, E.V. 
\textit{Higher regularity of solutions to fully nonlinear elliptic equations}. \url{arXiv:2312.02932v2}.

\bibitem{PSSW09} Peres, Y., Schramm, O., Sheffield,  S. and   Wilson, D.B.
\textit{Tug-of-war and the infinity Laplacian}. J. Amer. Math. Soc., 22 (1) : 167-210, 2009.

%\bibitem{Tei14} Teixeira, E.V.
%\textit{Regularity for quasilinear equations on degenerate singular sets}.
%Math. Ann. 358 (2014), no.1-2, 241-256.

\bibitem{Tei16} Teixeira,  E.V.
\textit{Regularity for the fully nonlinear dead-core problem}.
Math. Ann. 364 (2016), no. 3-4, 1121-1134.

\bibitem{Tei22}  Teixeira, E.V.
\textit{On the critical point regularity for degenerate diffusive equations}.
Arch. Ration. Mech. Anal. 244 (2022), no.2, 293–316.

\bibitem{RTU15} Rossi, J.D., Teixeira, E.V. and  Urbano, J.M.,
\textit{Optimal regularity at the free boundary for the infinity obstacle problem}. Interfaces Free Bound. 17 (2015), no.3, 381–398.

\bibitem{Savin05} Savin, O.
\textit{$C^1$ regularity for infinity harmonic functions in two dimensions}. Arch. Ration. Mech. Anal. 176 (2005), no.3, 351–361.

\bibitem{Yu06} Yu, Y.  \textit{A remark on $C^2$ infinity-harmonic functions}, Electronic Journal of Differential Equations 2006 no 122 (2006), pp. 1-4.

%\bibitem{Zaj} Zaj\'{i}\u{c}ek, L.
%\textit{Porosity and $\sigma-$porosity},
%Real Anal. Exchange, 13 (1987/88), 314-350.

\end{thebibliography}
\end{document}